\documentclass[11pt]{article}

\usepackage[T1]{fontenc}
\usepackage[utf8]{inputenc}

\usepackage{fullpage}

\usepackage{amsmath, amsthm, amssymb, mathtools, mathrsfs}
\numberwithin{equation}{section}
\usepackage{hyperref}
\usepackage{xcolor}

\usepackage{titlesec}
\renewcommand{\thesection}{\arabic{section}}
\renewcommand{\thesubsection}{\thesection.\arabic{subsection}}
\titleformat{\section}{\normalfont\Large\bfseries}{\thesection.}{0.5em}{}
\titleformat{\subsection}{\normalfont\large\bfseries}{\thesubsection.}{0.5em}{}
\titleformat{\subsubsection}{\normalfont\normalsize\bfseries}{\thesubsubsection.}{0.5em}{}

\DeclareMathOperator{\supp}{supp}
\DeclareMathOperator{\inj}{inj}

\newtheorem{thmcounter}{dummy}

\newtheorem{proposition}[thmcounter]{Proposition}

\makeatletter
\def\th@plain{%
  \thm@notefont{\bfseries}%
  \itshape
  \thm@headpunct{.}%
}
\def\th@definition{%
  \thm@notefont{\bfseries}%
  \normalfont
  \thm@headpunct{.}%
}
\makeatother

\theoremstyle{plain}
\newtheorem{Theorem}[thmcounter]{Theorem}

\theoremstyle{definition}
\newtheorem{Definition}[thmcounter]{Definition}
\newtheorem{Remark}[thmcounter]{Remark}

\usepackage{graphicx}
\usepackage{mathrsfs}
\usepackage[nottoc]{tocbibind}

\usepackage{graphicx}
\usepackage{mathrsfs}
\usepackage{amsmath, amssymb}

\usepackage{abstract}

\title{\vspace{-1cm} 
\textbf{Compactness of Conformal Metrics with $L^p$-Bounded Q-Curvature on Closed Smooth Riemannian Manifolds}
}

\author{%
Zeinab Mcheik\thanks{PhD student, Jean Leray Mathematics Laboratory, Université de Nantes, Nantes, France; and Laboratory of Analysis and Applied Mathematics (LAMA), Paris, France. Email: \texttt{zeinab.mcheik@univ-nantes.fr}}%
}

\date{} 

\begin{document}

\maketitle

 \vspace{4cm}   

\begin{abstract}
Let \((M^n,g)\) be a smooth closed Riemannian manifold of dimension \(n \ge 5\) with positive Yamabe invariant and semi-positive \(Q\)-curvature. We establish a precompactness result in the \(C^{\alpha}\)-Hölder topologie on the space of Riemannian metrics, for some \(\alpha>0\),  for the set of  metrics \(\tilde{g}\) conformal to \(g\), with volume equal to that of the standard sphere $\mathbb{S}^n$, whose \(Q\)-curvature is nonnegative and uniformly bounded in \(L^p(M,\tilde{g})\) for some \(p > \frac{n}{4}\), and whose first positive eigenvalue of the Laplace-Beltrami operator satisfies \( \lambda_1(M,\tilde{g}) \ge n + \frac{1}{\Lambda} \)
for some positive constant \(\Lambda\).
\end{abstract}
\medskip


\newpage

\maketitle
\tableofcontents
\newpage

\section{Introduction}

~~ Let $(M^n,g)$ be a smooth, closed Riemannian manifold. In dimension $n \ge 5$, the natural fourth-order analog of the scalar curvature is the \emph{$Q$-curvature}, introduced by Branson~\cite{zbMATH03960440}. It is defined by \[ Q_g = -a_n \Delta_g S_g + b_n S_g^2 - c_n |{\rm Ric}_g|^2, \] where $S_g$ denotes the scalar curvature, ${\rm Ric}_g$ the Ricci tensor, and $a_n,b_n,c_n$ positive constants depending only on the dimension $n$. Associated with the \(Q\)-curvature is the \emph{Paneitz operator} \cite{paneitz2008quartic}, a fourth-order conformally covariant differential operator defined by \[ P_g u = \Delta_g^2 u + \operatorname{div}_g\!\big( a_n' S_g\, g - b_n' \operatorname{Ric}_g \big)(\nabla u,\cdot) + \frac{n-4}{2}\, Q_g\, u, \] where \(a_n', b_n' \) are positive constants depending only on the dimension \(n\). For a smooth function \(\varphi\) and a smooth positive function \(u\), the Paneitz operator and the \(Q\)-curvature transform under the conformal change of the metric \(g_u = u^{\frac{4}{n-4}} g\) according to \begin{align}
\label{pan_pde}
P_g (u \varphi) &= P_{g_u}(\varphi)\, u^{\frac{n+4}{n-4}},\\[2mm]
\label{paneitz_eq}
P_g u &= \frac{n-4}{2}\, Q_{g_u}\, u^{\frac{n+4}{n-4}}.
\end{align} Under suitable positivity assumptions on the scalar curvature and the $Q$-curvature, Gursky and Malchiodi \cite{gursky2015strong} established the existence of a smooth metric in the conformal class of $g$ with \emph{constant positive $Q$-curvature}. More precisely, if
\[
S_g > 0, \quad Q_g \ge 0 \quad \text{ and } ~ Q_g > 0 \text{ somewhere},
\]
then there exists a smooth positive function $u$ solving the positive constant $Q$-curvature equation
\begin{equation}\label{p123}
    P_g u = \lambda\, u^{\frac{n+4}{n-4}}, \qquad \lambda\in \mathbb{R}^{*}_+,
\end{equation}
so that the metric $g_u=u^\frac{4}{n-4}g$ has a positive constant Q-curvature, and $\lambda$ is  determined by
\[\lambda=\frac{n-4}{2} Q_{g_u}.
\]
In this paper, we study the compactness of conformal metrics of the form
\(
g_u := u^{\frac{4}{n-4}}\, g,\)
assuming that, for some \(p > \frac{n}{4}\), there exists a positive constant \(C\), independent of \(u\), such that
\[
\|Q_{g_u}\|_{L^p(M,g_u)} \le C.
\]
The precise assumptions on \(u\) are detailed in Theorem~\ref{theorem 1}.

\medskip

Because of the conformal invariance of the problem, it is instructive to recall the classical second-order Yamabe problem~\cite{zbMATH04030435}, which exhibits closely related phenomena. On a closed Riemannian manifold $(M^n,g)$ with $n \ge 3$, the problem of prescribing constant scalar curvature—viewed as a higher-dimensional generalization of the Uniformization Theorem~\cite{zbMATH05835819}—amounts to finding a metric $h$ conformal to $g$ whose scalar curvature $S_h$ is constant. By conformal invariance, this problem reduces to finding a smooth positive function $u$ satisfying the constant scalar curvature equation
\begin{equation}\label{yam}
L_g u = \beta\, u^{\frac{n+2}{n-2}}, \quad \beta \in \mathbb{R},
\end{equation}
where \( L_g = -4\,\frac{n-1}{n-2}\,\Delta_g + S_g \) denotes the Yamabe operator, and $S_g$ is the scalar curvature of $(M,g)$. Then the metric $u^{\frac{4}{n-2}} g$ has a constant scalar curvature equal to $\beta$.

\medskip

To study the compactness of the family of solutions to the Yamabe equation~\eqref{yam}, we rely on the work of Khuri, Marques, and Schoen~\cite{zbMATH05503073}, who proved that on a closed Riemannian manifold \((M^n,g)\) of dimension \(3 \le n \le 24\),  if \(M\) is not conformally diffeomorphic to the standard sphere \(\mathbb{S}^n\) and that the Yamabe invariant of $(M^n,g)$ is positive, then the set of all smooth positive solutions \(u\) of~\eqref{yam} is relatively compact in \(C^2(M)\). Similar results in lower dimensions were obtained by Druet~\cite{zbMATH02207680} for \(n \le 5\), by Marques~\cite{zbMATH05033801} and Li–Zhang~\cite{zbMATH02220860} for \(n \le 7\), and by Li–Zhang~\cite{zbMATH05161231} for \(n \le 11\). In higher dimensions, compactness may fail: Brendle~\cite{zbMATH05817575} constructed counterexamples for \(n \ge 52\), and Brendle and Marques~\cite{zbMATH05522838} extended these examples to \(25 \le n \le 51\).

\medskip

In the study of compactness for families of conformal metrics of the form \(g_u = u^{\frac{4}{n-2}} g,\) without assuming that the scalar curvature \(S_{g_u}\) is constant, Gursky~\cite{zbMATH00558966} introduced a \emph{mass concentration criterion} that controls the number of bubbles and plays a fundamental role in compactness results modulo bubbling. More precisely, let \((M^n,g)\) be a smooth closed Riemannian manifold of dimension \(n \ge 3\). Assume that \(M\) is not conformally equivalent to the standard sphere, that the Riemann curvature tensor of the metrics \(g_u\) is uniformly bounded in  \(L^p(M,g_u)\)  for some \(p>n/2\), and that the volumes of the metrics \(g_u\) are uniformly controlled, i.e., there exists a constant \(C>0\) independent of \(u\) such that
\(C^{-1} \le \mathrm{Vol}(M,g_u) \le C.\) Under these assumptions, Gursky proved that the family \(\{g_u\}\) is precompact. More recently, Matthiesen~\cite{zbMATH06774840} obtained a related compactness result under \emph{weaker assumptions}. More precisely, he considered families of conformal metrics \(g_u = u^{\frac{4}{n-2}} g\) such that the scalar curvature is uniformly bounded in \(L^p(M,g_u)\) for some \(p > \frac{n}{2}\), the first positive eigenvalue of the Laplacian associated with \(g_u\) satisfies \(\lambda_1(M,g_u) \ge n + \frac{1}{\Lambda}\) for some constant \(\Lambda > 0\), and that the volume of the manifold \(( M,g_u)\) is equal to that of the standard sphere $\mathbb{S}^n$, without any control on the full curvature tensor. A key tool in Matthiesen's compactness results is the \emph{variational characterization of the first positive eigenvalue} of the Beltrami-Laplacian \(\lambda_1(M,g)\) on a closed manifold \(M\):
\begin{equation}\label{rlq}
    \lambda_1(M,g) = \inf_{\substack{u \in H^1(M)\setminus \{0\} \\ \int_M u\, dV_g = 0}} 
\frac{\displaystyle \int_M |\nabla u|_g^2 \, dV_g}{\displaystyle \int_M u^2 \, dV_g}.
\end{equation}
 By Cheeger’s inequality \cite{zbMATH03337135}, \(\lambda_1(M,g)\) measures the connectivity of \((M,g)\). A positive lower bound on \(\lambda_1(M,g)\) prevents the manifold from splitting into two large regions connected by a narrow tube, and in particular rules out the formation of two distinct blow-up points.
\medskip
 
 In the Paneitz setting, a similar compactness question arises for conformal metrics of the form $g_u = u^{\frac{4}{n-4}} g$ . Assume that $(M^n,g)$ is not conformally equivalent to the standard sphere and that
\[
S_g \ge 0 \quad \text{and} \quad Q_g \ge 0, \qquad \text{with } Q_g>0 \text{ somewhere on } M.
\]
Under these assumptions, Gang Li~\cite{Gangli} proved that, in dimensions $5 \le n \le 7$, the set of metrics conformal to $g$ with a fixed positive constant $Q$-curvature is relatively compact in $C^{4,\alpha}(M)$ for any $0<\alpha<1$. The proof relies on a refined blow-up analysis combined with Pohozaev-type identities, which rule out bubbling phenomena except in the spherical case. More general compactness results for conformal metrics with constant $Q$-curvature, under additional geometric assumptions, were later obtained by Y. Li and J. Xiong~\cite{yanyanli}.

\medskip

 Our first result establishes a compactness property for the conformal factors $u$ associated with metrics of the form $g_u = u^{\frac{4}{n-4}} g$. This result is obtained under the assumptions that the $Q$-curvature $Q_{g_u}$ is uniformly bounded in $L^p(M,g_u)$ for some $p > \frac{n}{4}$, that the first positive eigenvalue $\lambda_1(M,g_u)$ of the Laplace-Beltrami operator associated with $g_u$ admits a uniform positive lower bound, and that the volume of $g_u$ is fixed and positive. More precisely, let \((M^n,g)\) be a smooth closed Riemannian manifold of dimension \(n \ge 5\). For \(p > \frac{n}{4}\) and \(\Lambda > 0\), we define the class of functions \(\mathcal{E}_{\Lambda,p}\) by
\begin{equation}\label{function class}
\mathcal{E}_{\Lambda,p} 
:= \Bigl\{
u \in C^\infty(M),\ u>0,\ 
\int_M u^{\frac{2n}{n-4}}\, dV_g = \sigma_n, \ 
\left\| Q_{u^{\frac{4}{n-4}} g} \right\|_{L^p(M,u^{\frac{4}{n-4}} g)} \le \Lambda, \ 
\lambda_1\Bigl(M,{u^{\frac{4}{n-4}} g}\Bigr) \ge n + \frac{1}{\Lambda} 
\Bigr\},
\end{equation}
where \(\sigma_n\) denotes the volume of the unit Euclidean sphere \(\mathbb{S}^n\).

\begin{Theorem}[Compactness of conformal factors]\label{theorem 1}
The family \(\mathcal{E}_{\Lambda,p}\) is relatively compact in \(C^\alpha(M)\) for any $\alpha$ satisfying
\[
0 < \alpha < \min\Bigl\{1, 4-\frac{n}{p}\Bigr\}.
\]
\end{Theorem}

\medskip

Let $u$ be a function in $\mathcal{E}_{\Lambda,p}$, and denote by $g_u = u^{\frac{4}{n-4}} g$ the corresponding conformal metric on M.   Throughout this work, we impose volume normalization
\begin{equation}\label{volume_cndt}
   \int_M u^{\frac{2n}{n-4}} \, dV_g = \sigma_n .
\end{equation}
More generally, one may relax condition~\eqref{volume_cndt} and only assume the existence of a positive constant $C$, independent of $u$, such that
\[
C^{-1} \le \int_M u^{\frac{2n}{n-4}} \, dV_g \le C.
\]
In this case, the assumption on the first positive eigenvalue of the Laplacian must be replaced by
\begin{equation*}\label{lambda1_vol_cndt}
\lambda_1(M,g_u) \, \mathrm{Vol}(M,g_u)^{\frac{2}{n}} \ge (n+\frac{1}{\Lambda})\, \sigma_n^{\frac{2}{n}}.
\end{equation*}
 Imposing the stricter normalization~\eqref{volume_cndt} simplifies several notations and will be adopted throughout the paper.

\medskip

The main idea underlying Theorem~\ref{theorem 1} is that, in view of the defining properties of the set $\mathcal{E}_{\Lambda,p}$, sequences in this class cannot blow-up.
More precisely, for any sequence $(u_k)_k \subset \mathcal{E}_{\Lambda,p}$, there exists a positive constant $C$, independent of $k$, such that
\begin{equation}\label{inf1}
\|u_k\|_{L^{\infty}(M)} \le C.
\end{equation}
The proof of~\eqref{inf1} proceeds by contradiction. Assuming the existence of a blow-up sequence, $\Vert u_k\Vert \rightarrow +\infty$, we select a sequence of concentration points $x_k \in M$ such that
\[
u_k(x_k)=\|u_k\|_{L^\infty(M)}.
\]
   Performing a suitable rescaling around the points $x_k$, one constructs a sequence of functions $(v_k)_k$ defined on $\mathbb{R}^n$ that satisfy uniform bounds in $W^{4,p}_{\mathrm{loc}}(\mathbb{R}^n)$ (see Proposition \ref{prop5}). Up to extracting a subsequence, $(v_k)_k$ converges in $W^{4,p}_{\mathrm{loc}}(\mathbb{R}^n)$ to a nonnegative limiting function $v_\infty$. Let $\operatorname{inj}(M)$ denote the injectivity radius of $(M,g)$. Let $h_k$ be the rescaled metrics around $x_k \in M$, defined by
\[
h_k(y) := (\exp_{x_k}^* g)(\mu_k y), \qquad y \in B(0,r) \subset \mathbb{R}^n,
\]
where $\mu_k = \|u_k\|_{\infty}^{-\frac{2}{n-4}}$ is the positive scaling factor, $r>0$ is fixed, and $k$ is sufficiently large so that $\mu_k r \leq \operatorname{inj}(M)$. Then, one has the weak convergence of measures
\[
v_k^{\frac{2n}{n-4}} \, dV_{h_k} \;\rightharpoonup\; v_\infty^{\frac{2n}{n-4}} \, dy ,
\]
which, in general, does not exclude the possibility of mass loss at infinity. However, the existence of a positive constant $C$, independent of $k$, such that
\[
\lambda_1(M,g_{u_k}) \ge C, \qquad g_{u_k}= u_k^{\frac{4}{n-4}}g,
\]
 prevents any splitting of mass on the manifold, since such a splitting would force the first  positive eigenvalue to approach zero. As a consequence, no mass can be lost at infinity in the blow-up limit, and the total volume is entirely captured by the limiting profile $v_\infty$. More precisely, one obtains the volume preservation property (see Proposition~\ref{contr1})
\begin{equation}\label{volrese}
\int_{\mathbb{R}^n} v_\infty^{\frac{2n}{n-4}} \, dy = \sigma_n.
\end{equation}

Once the volume preservation property \eqref{volrese} is established, the limiting measure $v_\infty^{\frac{2n}{n-4}}\,dy$ defines a finite measure on $\mathbb{R}^n$ whose total mass coincides with that of the standard sphere. Hence, the Hersch-type estimate (see Lemma~2.5 in~\cite{zbMATH06774840} or property (2.4) in \cite{zbMATH00027724}) yields the uniform upper bound
\begin{equation}\label{hersh}
\inf_{\substack{\phi \in C_c^\infty(\mathbb{R}^n) \\[1mm]
\int_{\mathbb{R}^n} \phi\, v_\infty^{\frac{2n}{n-4}} \, dy = 0}}
\frac{\displaystyle \int_{\mathbb{R}^n} v_\infty^{\frac{2(n-2)}{n-4}} |\nabla \phi|^2 \, dy}
{\displaystyle \int_{\mathbb{R}^n} v_\infty^{\frac{2n}{n-4}} \phi^2 \, dy}
\le n,
\end{equation}
see also Proposition~\ref{proposition 1}. On the other hand,  since the sequence $(u_k)_k$ belongs to the class $\mathcal{E}_{\Lambda,p}$, the associated conformal metrics satisfy
\[
\lambda_1(M,{{u_k^\frac{4}{n-4}}}g) \ge n+\frac{1}{\Lambda}.
\]
Then, passing to the limit in Rayleigh's quotient \eqref{rlq} yields
\begin{equation}\label{rayleigh}
\inf_{\substack{\phi \in C_c^\infty(\mathbb{R}^n) \\[1mm]
\int_{\mathbb{R}^n} \phi\, v_\infty^{\frac{2n}{n-4}} \, dy = 0}}
\frac{\displaystyle \int_{\mathbb{R}^n} v_\infty^{\frac{2(n-2)}{n-4}} |\nabla \phi|^2 \, dy}
{\displaystyle \int_{\mathbb{R}^n} v_\infty^{\frac{2n}{n-4}} \phi^2 \, dy}
\ge n + \frac{1}{\Lambda},
\end{equation}
which contradicts \eqref{hersh}. This contradiction rules out the occurrence of blow-up and establishes the desired uniform $L^\infty$ bound \eqref{inf1}. Further details of the blow-up analysis are provided in Section~\ref{blup}.

\medskip

Our second result concerns the convergence of the sequence of conformal metrics $(u_k^{\frac{4}{n-4}}g)_k$, where $(u_k)_k$ is a sequence in $\mathcal{E}_{\Lambda,p}$. When $(u_k)_k$ converges, it is essential to ensure that the limit remains strictly positive on $M$, so that it defines a Riemannian metric. In the classical Yamabe problem, the Harnack  inequality \cite{trudinger1968remarks} guaranties that an uniform upper bound on the conformal factor \(u\) also provides a uniform positive lower bound. However, this is no longer the case in our setting, since the Paneitz operator is a fourth-order elliptic operator, rather than a second-order one. This difficulty motivates the introduction of additional geometric assumptions, namely the positivity of the scalar curvature and the semi-positivity of the \(Q\)-curvature, which, thanks to Gursky and Malchiodi~\cite{gursky2015strong}, allows us to establish a positive uniform lower bound for the function $u$ in $\mathcal{E}_{ \Lambda,p}$ and leads to the following result.

\medskip

\begin{Theorem}[Compactness of conformal metrics]\label{theorem 2}
Let $(M^n,g)$ be a smooth closed Riemannian manifold of dimension $n \ge 5$, and assume that the scalar curvature and the $Q$-curvature satisfy
\[
S_g \ge 0 \quad \text{and} \quad Q_g \ge 0 \quad \text{with } Q_g > 0 \text{ somewhere on } M.
\]

Fix $p > \frac{n}{4}$ and $\Lambda > 0$. Let $\mathcal{E}_{\Lambda,p}$ be the class of functions defined in \eqref{function class}, and define the corresponding set of conformal metrics $E_{\Lambda,p}$ by
\begin{equation}\label{def:E_Lambda_p}
E_{\Lambda,p} := \bigl\{\, g_u = u^\frac{4}{n-4} g \ :\ u \in \mathcal{E}_{\Lambda,p} \text{ and } P_g(u) \ge 0 \,\bigr\}.
\end{equation}

For any $\alpha \in (0,1)$, we also define the Hölder space of conformal metrics
\[
\mathscr{C}^{\alpha}(M) := \bigl\{\, g_u = u g \ :\ u \text{ is a positive function on } M \text{ with } u \in C^{\alpha}(M) \,\bigr\}.
\]

Then $E_{\Lambda,p}$ is relatively compact in $\mathscr{C}^{\alpha}(M)$ for any $\alpha$ satisfying
\[
0 < \alpha < \min\Bigl\{1,\, 4 - \frac{n}{p}\Bigr\}.
\]
\end{Theorem}

\medskip

The paper is organized as follows. Section~\ref{sec proof} contains the blow-up analysis and the derivation of a uniform \(L^\infty\) bound, which together yield Hölder compactness and lead to the proof of Theorem~\ref{theorem 1}. Finally, Section~\ref{sec3} is dedicated to the proof of Theorem~\ref{theorem 2} and presents a counterexample illustrating the necessity of the sign condition \(P_g(u)\geq 0\), which is part of the definition of the set \({E}_{\Lambda,p}\), in the proof of that result.

\medskip

\medskip
\subsection*{Notation}

Throughout this paper, we adopt the following notations:

\begin{itemize}
   \item The symbol $\delta$ denotes the standard Euclidean metric on $\mathbb{R}^n$. The metric induced by $\mathbb{R}^{n+1}$ on the unit sphere $\mathbb{S}^n$ is denoted by $\bar{g}$.

    \item Let $(M,g)$  be a Riemannian manifold. For any measurable subset \( A \subset M \), we denote by \( V_g(A) \) its Riemannian volume with respect to \( g \):
    \[
    V_g(A) := \int_A dV_g,
    \]
    where \( dV_g \) is the volume form associated with \( g \).

 \item \( C \) denotes a positive constant that may change from line to line and depends only on \( n \), \( M \), and \( g \).\end{itemize}

\medskip

\section*{Acknowledgement}

I would like to sincerely thank my supervisors, Gilles Carron and Paul Laurain, for their guidance, support, and many insightful discussions throughout this work. Their expertise and encouragement were invaluable to the development of this research.

Funded by the European Union. Views and opinions expressed are, however, those of the author only and do not necessarily reflect those of the European Union or the European Research Executive Agency. Neither the European Union nor the granting authority can be held responsible for them.

\section{Proof of Theorem \ref{theorem 1}}\label{sec proof}
In this section, we follow the strategy outlined in the introduction to show that any sequence of functions $(u_k)_k$ in $\mathcal{E}_{\Lambda,p}$ is uniformly bounded in $L^\infty(M)$. The rest of the proof of Theorem~\ref{theorem 1} then relies on standard Sobolev estimates.

\subsection{Spectral Estimate}\label{ltasmiyet}
Here, we introduce a variational quantity associated with a nonnegative function on \(\mathbb{R}^n.\)

\begin{Definition}\label{def}
Let $u$ be a nontrivial, nonnegative, continuous function on $\mathbb{R}^n$ such that 
\(
u \in L^{\frac{2n}{n-2}}(\mathbb{R}^n).\) We define
\[
\lambda_1(\mathbb{R}^n,g_u) := 
\inf_{\substack{v \in C_c^{\infty}(\mathbb{R}^n) \\ \int_{\mathbb{R}^n} v \, u^{\frac{2n}{n-2}}\, dy = 0}} 
\frac{\displaystyle \| u \nabla v \|_{L^2(\mathbb{R}^n)}^2}{\displaystyle \| u^{\frac{n}{n-2}} v \|_{L^2(\mathbb{R}^n)}^2}.
\]
\end{Definition}
Note that \(\lambda_1(\mathbb{R}^n,g_u) \) does not correspond to a classical eigenvalue of the Laplacian, since \(\mathbb{R}^n\) is non-compact. 
However, if \(u\) is strictly positive on \(\mathbb{R}^n\), one can consider the Friedrichs extension of the Laplacian associated with the conformal metric
\(g_u := u^{\frac{4}{n-2}} \, \delta,\) 
denoted by \(\Delta_{g_u}\). In this case, \(\lambda_1(\mathbb{R}^n,g_u)\) represents the distance between the zero eigenvalue and the rest of the spectrum of \(\Delta_{g_u}\).

\medskip

The aim of this section is to establish the following Hersch-type upper bound.

\begin{proposition}\label{proposition 1}
Let \(n \geq 3\) and \(u\) be nontrivial,   continuous, non-negative functions in \( L^{\frac{2n}{n-2}}(\mathbb{R}^n)\). Assume that \(u\) satisfies the volume normalization
\begin{equation}\label{1}
\int_{\mathbb{R}^n} u^{\frac{2n}{n-2}} \, dy = \sigma_{n}.
\end{equation}
Then the quantity \(\lambda_1(\mathbb{R}^n,g_u) \) given in Definition~\ref{def} satisfies the upper bound
\[
\lambda_1(\mathbb{R}^n,g_u) \leq n.
\]
\end{proposition}

\proof

Since $u \ge 0$ and $u^{\frac{2n}{n-2}} \in L^1(\mathbb{R}^n)$, the measure $u^{\frac{2n}{n-2}}\,dy$ defines a finite measure on $\mathbb{R}^n$ satisfying
\[
0 < \int_{\mathbb{R}^n} u^{\frac{2n}{n-2}} \, dy < \infty.
\]

By an argument analogous to the conformal uniformization of Hersch type 
(see Lemma 2.5 in \cite{zbMATH06774840} or property (2.4) in \cite{zbMATH00027724}), 
there exists a smooth conformal map
\[
\psi \colon (\mathbb{R}^n, u^{\frac{2n}{n-2}} \, dy) \longrightarrow (\mathbb{S}^n, dV_{\bar g})
\]
such that
\begin{equation}\label{eq:center}
\int_{\mathbb{R}^n} (x^i \circ \psi) \, u^{\frac{2n}{n-2}} \, dy = 0, 
\qquad 1 \le i \le n+1.
\end{equation}

\medskip

Let \(0 < r < R\), and define a radial cut-off function \(\eta_R \in W^{1,n}(\mathbb{R}^n)\) by
\[
\eta_R(x) :=
\begin{cases}
1, & |x| < r,\\[2mm]
\displaystyle \frac{\log(|x|/R)}{\log(r/R)}, & r \le |x| \le R,\\[1mm]
0, & |x| > R.
\end{cases}
\]
Then \(\eta_R \to 1\) uniformly on compact subsets of \(\mathbb{R}^n\) as \(R \to \infty\), and for any \(f \in L^1(\mathbb{R}^n)\), the dominated convergence theorem implies
\begin{equation}\label{etar}
\int_{\mathbb{R}^n} f \, \eta_R \, dy \longrightarrow \int_{\mathbb{R}^n} f \, dy \quad \text{as } R \to \infty.
\end{equation}
A direct computation also gives
\begin{equation}\label{eq:cutoff-gradient}
\int_{\mathbb{R}^n} |\nabla \eta_R|^n \, dy
= \sigma_{n-1} \left( \log \left(\frac{R}{r} \right)\right)^{1-n}.
\end{equation}

\medskip
For each \( i\in \{1,...,n+1\}\)
we define 
\[
f_{i,R} := \eta_R \bigl(x^i \circ \psi - c_{i,R}\bigr), 
\qquad
c_{i,R} := \frac{\displaystyle \int_{\mathbb{R}^n} \eta_R (x^i \circ \psi) \, u^{\frac{2n}{n-2}} \, dy}{\displaystyle \int_{\mathbb{R}^n} \eta_R \, u^{\frac{2n}{n-2}} \, dy}.
\]
The function \(f_{i,R}\) lies in \(W^{1,n}(\mathbb{R}^n)\).
After regularizing  $f_{i,R}$ and applying the definition of \( \lambda_1(\mathbb{R}^n,g_u)\) we obtain 
\[
\lambda_1(\mathbb{R}^n, g_u) 
\int_{\mathbb{R}^n} u^{\frac{2n}{n-2}} f_{i,R}^2 \, dy
\le 
\int_{\mathbb{R}^n} u^2 |\nabla f_{i,R}|^2 \, dy.
\]

Summing over \(i = 1, \dots, n+1\) gives
\begin{equation}\label{eq:main-ineq}
\lambda_1(\mathbb{R}^n,g_u) \int_{\mathbb{R}^n} u^{\frac{2n}{n-2}} \sum_{i=1}^{n+1} f_{i,R}^2 \, dy
\le \int_{\mathbb{R}^n} u^2 \sum_{i=1}^{n+1} |\nabla f_{i,R}|^2 \, dy.
\end{equation}

We analyze the behavior of both sides of \eqref{eq:main-ineq} as \(R \to \infty\).

We first consider the left-hand side. By definition,
\[
\sum_{i=1}^{n+1} f_{i,R}^2
=
\eta_R^2
\left(
\sum_{i=1}^{n+1} (x^i \circ \psi)^2
- 2 \sum_{i=1}^{n+1} c_{i,R} (x^i \circ \psi)
+ \sum_{i=1}^{n+1} c_{i,R}^2
\right).
\]

From \eqref{etar} together with \eqref{eq:center}, we obtain
\[
\lim_{R \to \infty} c_{i,R}
=
\frac{\displaystyle \int_{\mathbb{R}^n} (x^i \circ \psi)\, u^{\frac{2n}{n-2}} \, dy}
{\displaystyle \int_{\mathbb{R}^n} u^{\frac{2n}{n-2}} \, dy}
= 0 .
\]

Since \(\eta_R \to 1\) pointwise, \(c_{i,R} \to 0\) as \(R \to \infty\), and \(x^i \circ \psi\) is bounded, it follows that
\[
\left(\int_{\mathbb{R}^n} u^{\frac{2n}{n-2}}
\sum_{i=1}^{n+1} f_{i,R}^2 \, dy\right)
\longrightarrow
\sigma_n
\qquad \text{as } R \to \infty .
\]

We now turn to the right-hand side of \eqref{eq:main-ineq}. Expanding the gradient term gives
\begin{align*}
|\nabla f_{i,R}|^2
&= |\nabla \eta_R|^2 (x^i \circ \psi - c_{i,R})^2 \\
&\quad + \eta_R^2 |\nabla(x^i \circ \psi)|^2 \\
&\quad + 2 \eta_R (x^i \circ \psi - c_{i,R})
\langle \nabla \eta_R, \nabla(x^i \circ \psi) \rangle .
\end{align*}

Summing over \(i\) and integrating, we obtain
\[
\int_{\mathbb{R}^n} u^2
\sum_{i=1}^{n+1} |\nabla f_{i,R}|^2 \, dy
= I_1 + I_2 + I_3 .
\]
with:
\begin{align*}
I_1 &= \int_{\mathbb{R}^n} u^2 |\nabla \eta_R|^2 \sum_{i=1}^{n+1} (x^i \circ \psi - c_{i,R})^2 \, dy, \\
I_2 &= \sum_{i=1}^{n+1} \int_{\mathbb{R}^n} u^2 \eta_R^2 |\nabla(x^i \circ \psi)|^2 \, dy, \\
I_3 &= 2 \sum_{i=1}^{n+1} \int_{\mathbb{R}^n} u^2 \eta_R (x^i \circ \psi - c_{i,R}) \langle \nabla \eta_R, \nabla(x^i \circ \psi) \rangle \, dy.
\end{align*}

Now we study the limit of  \( I_1 \), \( I_2 \), and \( I_3 \) as \( R \to \infty \).

\medskip

\noindent
\textbf{Estimate of \( I_1 \).}  
Since the functions \( x^i \circ \psi \) and the coefficients \( c_{i,R} \) are uniformly bounded, there exists a positive constant \( C \),  independent of $R$, such that
\[
I_1 \leq C \int_{\mathbb{R}^n} u^2 |\nabla \eta_R|^2 \, dy.
\]
 Applying Hölder's inequality with exponents \( \frac{n}{2} \) and \( \frac{n}{n-2} \), and using  \eqref{1}, we obtain
\begin{align*}
I_1 &\leq C \left( \int_{\mathbb{R}^n} u^{\frac{2n}{n-2}} \, dy \right)^{\frac{n-2}{n}} \left( \int_{\mathbb{R}^n} |\nabla \eta_R|^n \, dy \right)^{\frac{2}{n}} \\
&= C \sigma_n^{\frac{n-2}{n}} \left( \sigma_{n-1} \left( \log\left( \frac{R}{r} \right) \right)^{1-n} \right)^{\frac{2}{n}} \\
&= C \sigma_n^{\frac{n-2}{n}} \sigma_{n-1}^{\frac{2}{n}} \left( \log\left( \frac{R}{r} \right) \right)^{\frac{2(1-n)}{n}} \to 0 \quad \text{as } R \to \infty.
\end{align*}
Therefore, \[\lim_{R\to \infty} I_{1}=0.\]
\medskip

\noindent
\textbf{Estimate of \(I_2\).}  
By the dominated convergence theorem and the pointwise convergence \(\eta_R \to 1\) as \(R \to \infty\), we have
\[
\lim_{R \to \infty} I_2
=
\sum_{i=1}^{n+1} \int_{\mathbb{R}^n} u^2 \, |\nabla(x^i \circ \psi)|^2 \, dy.
\]

Let 
\[
\Omega := \{ x \in \mathbb{R}^n : u(x) > 0 \}.
\] 
Since \(u = 0\) on \(\mathbb{R}^n \setminus \Omega\), the above integral can be restricted to \(\Omega\), where \(u\) is strictly positive. Therefore,  under the conformal metric change on $\Omega$ and for \(g_u = u^{\frac{4}{n-2}} \delta\), we obtain
\[
\sum_{i=1}^{n+1} \int_{\mathbb{R}^n} u^2 \, |\nabla(x^i \circ \psi)|^2 \, dy
=
\sum_{i=1}^{n+1} \int_{\Omega} |\nabla(x^i \circ \psi)|_{g_u}^2 \, dV_{g_u}.
\]

Since \(\psi\) is conformal, there exists a positive function \(f\) on \(\Omega\) such that
\begin{equation}\label{lali ta}
    \psi^* \bar g = f^2 g_u ,
\end{equation}
and hence
\[
\sum_{i=1}^{n+1} |\nabla(x^i \circ \psi)|_{g_u}^2 = n f^2.
\]

It follows that
\begin{align*}
\sum_{i=1}^{n+1} \int_{\mathbb{R}^n} u^2 \, |\nabla(x^i \circ \psi)|^2 \, dy
&= \sum_{i=1}^{n+1} \int_{\Omega} |\nabla(x^i \circ \psi)|_{g_u}^2 \, dV_{g_u} \\
&= n \int_{\Omega} f^2 \, dV_{g_u} \\
&\le n \left( \int_{\Omega} f^n \, dV_{g_u} \right)^{\frac{2}{n}}
      \left( \int_{\Omega} dV_{g_u} \right)^{\frac{n-2}{n}},
\end{align*}
where the last inequality follows from Hölder's inequality. Next, using  \eqref{lali ta} and the fact that \(\Omega \subset \mathbb{S}^n\), we obtain
\begin{align*}
\int_{\Omega} f^n \, dV_{g_u}
&= \int_{\Omega} dV_{\psi^*(\bar g)} \\
&= \int_{\psi(\Omega)} dV_{\bar g} \\
&\le \int_{\mathbb{S}^n} dV_{\bar g} \\
&= \sigma_n .
\end{align*}

Combining this estimate with the following volume normalization
\[
\int_{\Omega} dV_{g_u} = \sigma_n,
\]
we conclude that
\[
I_2 \le n \, \sigma_n .
\]
\noindent

\textbf{Estimate of \(I_3\).} 
 By the Cauchy-Schwarz inequality, there exists a positive constant \(C\), independent of $R$, such that
\[
I_3 \le C \sqrt{I_1 I_2}.
\]
Since \(I_2\) is bounded and \(I_1 \to 0\) as \(R \to \infty\), we deduce that
\[
\lim_{R \to \infty} I_3 = 0.
\]
Passing to the limit \( R \to \infty \) in the inequality \eqref{eq:main-ineq} and combining the above estimates, we deduce
\[
\sigma_n \, \lambda_1(\mathbb{R}^n,g_u) \leq n \, \sigma_n.
\]
$\hfill$ $\blacksquare$

\subsection{Volume control}\label{vl}
In this section, we derive estimates for the volumes of small and large geodesic balls under a conformal metric. The following estimate is inspired by the work of Matthiesen (see Proposition 2.4 in~\cite{zbMATH06774840}).

\begin{proposition}\label{prop:volume-inequality}
Let $(M^n, g)$ be a closed Riemannian manifold of dimension $n \geq 3$, and let $x \in M$. 
Let $u \in C^0(M)$ be a positive function satisfying the normalization
\[
\int_M u^{\frac{2n}{n-2}} \, dV_g = \sigma_n.
\]

Assume further that there exists a positive constant $\Lambda$ such that the first positive eigenvalue of the Laplacian of the conformal metric $g_u = u^{\frac{4}{n-2}} g$ satisfies
\[
\lambda_1(M, g_u) \geq n + \frac{1}{\Lambda}.
\]

Then, there exists a constant $C = C(n, \Lambda, M) > 0$ such that for all $0 < r < R < \operatorname{inj}(M)$, the following volume inequality holds:
\begin{equation}\label{eq:vol-inequality}
\bigl(V_{g_u}(B(x,r))\bigr)^2 \,\bigl(\sigma_n - V_{g_u}(B(x,R))\bigr)
\leq C \, \left( \log \left(\frac{R}{r}\right) \right)^{\frac{2(1-n)}{n}}.
\end{equation}
\end{proposition}

\proof
Let 
\(
g_u = u^{\frac{4}{n-2}} g
\) 
denote the conformal metric to g associated with \(u\).  
We define the cut-off function $\eta_{r,R}$ on \(M\) by
\[
\eta_{r,R}(z) =
\begin{cases}
\displaystyle \frac{\log \big( d_g(x,z) / R \big)}{\log (r / R)}, & \text{if } r \le d_g(x,z) \le R, \\[2mm]
1, & \text{if } d_g(x,z) \le r, \\[1mm]
0, & \text{if } d_g(x,z) \ge R,
\end{cases}
\]
where \(d_g\) denotes the geodesic distance with respect to the metric \(g\) on $M$. We also define the scalar
\[
\alpha_{r,R} := \int_M u^{\frac{2n}{n-2}} \, \eta_{r,R} \, dV_g
.\]

By construction, \(\eta_{r,R} \in C_c^\infty(M)\), with support contained in the geodesic ball \(B(x,R)\).  
There exists a positive constant \(C = C(M,g)\) such that
\[
\int_M |\nabla \eta_{r,R}|_g^n \, dV_g \le C \left(\log\left(R/r\right)\right)^{1-n}.
\]

The function \(\eta_{r,R}\) will serve as a localized test function for the conformal metric \(g_u\),  
while \(\alpha_{r,R}\) measures its weighted volume.  
This setup is crucial for the subsequent estimates.

By the variational characterization of the first positive eigenvalue of the Laplacian on a closed manifold \eqref{rlq}, it holds that
\[
\lambda_1(M, g_u) \int_M (\eta_{r,R} - \alpha_{r,R})^2 \, dV_{g_u} \leq \int_M u^2 |\nabla \eta_{r,R}|_g^2 \, dV_g.
\]

Since \(\eta_{r,R}\) vanishes outside the geodesic ball \(B(x,R)\), we can estimate the left-hand side from below by
\[
\int_{M \setminus B(x,R)} (\eta_{r,R} - \alpha_{r,R})^2 \, dV_{g_u} = \alpha_{r,R}^2 \left( \sigma_n - V_{g_u}(B(x,R)) \right).
\]
This implies
\[
\alpha_{r,R}^2 \left( \sigma_n - V_{g_u}(B(x,R)) \right) \leq \int_M (\eta_{r,R} - \alpha_{r,R})^2 \, dV_{g_u}.
\]

Combining this with the lower bound on \(\lambda_1(M,g_u)\), we obtain
\[
\alpha_{r,R}^2 \left(n + \frac{1}{\Lambda}\right) \left( \sigma_n - V_{g_u}(B(x,R)) \right) \leq \int_M u^2 |\nabla \eta_{r,R}|_g^2 \, dV_g.
\]

Noting that \(V_{g_u}(B(x,r)) \leq \alpha_{r,R}\), we deduce
\[
\left(n + \frac{1}{\Lambda}\right) \left[ V_{g_u}(B(x,r)) \right]^2 \left( \sigma_n - V_{g_u}(B(x,R)) \right) \leq \int_M u^2 |\nabla \eta_{r,R}|_g^2 \, dV_g.
\]

Applying Hölder's inequality and using the normalization condition on \(u\), we further estimate
\[
\int_M u^2 |\nabla \eta_{r,R}|_g^2 \, dV_g \leq \left( \int_M u^{\frac{2n}{n-4}} \, dV_g \right)^{\frac{n-4}{n}} \left( \int_M |\nabla \eta_{r,R}|_g^n \, dV_g \right)^{\frac{2}{n}} \leq C \left( \int_M |\nabla \eta_{r,R}|_g^n \, dV_g \right)^{\frac{2}{n}}.
\]

Finally, the quantity \(\int_M |\nabla \eta_{r,R}|_g^n \, dV_g\) is independent of the choice of the point $x$, and we have the following estimates

\[
\int_M |\nabla \eta_{r,R}|_g^n \, dV_g \leq \frac{C}{\left( \log \left(\frac{R}{r}\right) \right)^{n-1}},
\]
which yields the desired bound
\[
\int_M u^2 |\nabla \eta_{r,R}|_g^2 \, dV_g \leq \frac{C}{\left( \log \left(\frac{R}{r}\right) \right)^{\frac{2(n-1)}{n}}}.
\]

The inequality \eqref{eq:vol-inequality} then follows immediately.$\hfill$ $\blacksquare$
\subsection{Uniform $L^\infty$ bounds via blow-up analysis}

\subsubsection{Blow-up assumption and rescaling}\label{blup}
Let \((u_k)_k \subset \mathcal{E}_{\Lambda,p}\) be a sequence of functions.
Assume, for contradiction, that this sequence is unbounded in \(L^\infty(M)\), i.e.
\[
\|u_k\|_{L^\infty(M)} \longrightarrow +\infty \quad \text{as } k \to \infty.
\]

For each \(k\), let \(x_k \in M\) be a point where \(u_k\) attains its maximum:
\[
u_k(x_k) = \max_{M} u_k = \|u_k\|_{L^\infty(M)}.
\]

We introduce the scaling factor
\[
\mu_k := u_k(x_k)^{-\frac{2}{n-4}} \longrightarrow 0.
\]

Let \(r \) be a positive radius.  Since \(\mu_k \to 0\), we have \(\mu_k r < \operatorname{inj}(M)\) for \(k\) sufficiently large. 
We then use exponential coordinates centered at \(x_k\) and define the rescaled functions \(v_k\) by
\[
v_k(y) := \mu_k^{\frac{n-4}{2}}\, u_k\big( \exp_{x_k}(\mu_k y) \big),
\qquad y \in B\!\left(0,r\right).
\]

By construction,
\(
v_k(0) = 1 = \max_{B(0,r)} v_k.\) To analyze the geometry in these coordinates, we introduce the rescaling map
\[
\bar{\psi}_k(y) := \exp_{x_k}(\mu_k y), \qquad y \in B\!\left(0,r\right),
\]
and define the pullback metric
\[
(\bar{\psi}_k^*g)(y)
= \mu_k^2 (\exp_{x_k}^* g)(\mu_k y).
\]
We then normalize the metric by setting
\[
h_k(y) := \mu_k^{-2}\,(\bar{\psi}_k^*g)(y)
          = (\exp_{x_k}^* g)(\mu_k y),
          \qquad y \in B\left(0,r\right).
\]
Thus \(h_k\) is obtained by dilating the metric \(g\) around \(x_k\) with scale \(\mu_k\).
The rescaled metric \(h_k\) can be interpreted as a geometric \emph{zoom-in} around the blow-up point \(x_k\). The smoothness of the metric g on $M$ implies that
\[
h_k \longrightarrow \delta 
\quad \text{in } C^\infty_{\mathrm{loc}}(\mathbb{R}^n).
\]

\medskip

\noindent
\begin{Remark}[ Rescaled balls]
The map \(\bar{\psi}_k\) sends the Euclidean ball \(B\!\left(0,r\right)\) onto the geodesic ball \(B(x_k, \mu_kr)\) in \(M\):
\[
\bar{\psi}_k\left(B\!\left(0,r\right)\right) = B\left(x_k,\mu_k r\right).
\]
\end{Remark}

\noindent

\begin{Remark}[Relation between volume measures]
Since
\[
\bar{\psi}_k^*\!\big(u_k^{\frac{2n}{n-4}}\,dV_g\big)
   = v_k^{\frac{2n}{n-4}}\,dV_{h_k},
\]
for every integrable function \(f\) on \(M\), we have
\begin{equation}\label{bl trans}
   \int_{B(0, r)} (f \circ \bar{\psi}_k)\, v_k^{\frac{2n}{n-4}}\, dV_{h_k}
 = \!\int_{B(x_k, \mu_kr)}
    f\, u_k^{\frac{2n}{n-4}}\, dV_g. 
\end{equation}
This relation allows us to transfer the analysis of integrals near \(x_k\) to the domain \(B(0,r)\), and it will be crucial for the subsequent blow-up analysis.
\end{Remark}
\subsubsection{ Convergence to a limiting function}\label{hayde hiye}
\begin{proposition}\label{propo3}
For all $ k \in \mathbb{N}$, we denote by $g_{u_{k}}=u_{k}^{\frac{4}{n-4}}g$ the conformal metric to g on M with conformal factor  $ u_{k}^{\frac{4}{n-4}}$. The rescaled function $v_{k}$ in section \ref{blup} is positive and  uniformly bounded by \( v_k(0) = 1 \). It is a solution to the fourth order non-linear PDE:

 \begin{equation}
 P_{h_{k}}(v_{k})=(Q_{g_{u_{k}}}\circ{\bar{\psi_{k}} })v_{k}^{\frac{n+4}{n-4}},
\end{equation}

where $Q_{g_{u_{k}}}$ denotes the Q-curvature related to the metric $g_{u_{k}}$ on M.
\end{proposition}
\proof
The proof relies on the conformal invariance property of the Paneitz operator.
 In fact
\begin{align*}
P_{h_k}(v_k) 
&= \mu_k^{\frac{n+4}{2}} \, P_{\bar{\psi}_k^*( g)}(u_k \circ \bar{\psi}_k) \\
&= \mu_k^{\frac{n+4}{2}} \, \left (P_g(u_k) \right) \circ \bar{\psi}_k\\
&= (Q_{g_{u_k}} \circ \bar{\psi}_k) \, v_k^{\frac{n+4}{n-4}}.
\end{align*}
$\hfill$ $\blacksquare$
\noindent

\begin{proposition}\label{prop4}
The sequence $(P_{h_{k}}(v_{k}))_{k}$ is uniformly bounded in $L^{p}_{\mathrm{loc}}(\mathbb{R}^{n},h_{k})$.
\end{proposition}

\proof
By Proposition~\ref{propo3}, it suffices to show that the sequence
\[
\left(\left(Q_{g_{u_k}} \circ \bar{\psi}_k\right) \, v_k^{\frac{n+4}{n-4}} \right)_k
\]
is uniformly bounded in \(L^p(B(0,R), h_k)\) for every \(R >0\). Using relation \eqref{bl trans}, we obtain
\begin{align*}
\int_{B(0,R)} |Q_{g_{u_k}} \circ \bar{\psi}_k|^p v_k^{p\frac{n+4}{n-4}}\, dV_{h_k} 
&= \mu_k^{-n + \frac{p(n+4)}{2}} 
\int_{B(x_k,\mu_kR)} |Q_{g_{u_k}}|^p \, u_k^{p \frac{n+4}{n-4}} \, dV_g \\
&\le \|u_k\|_\infty^{p \frac{n+4}{n-4} - \frac{2n}{n-4}} 
\, \mu_k^{-n + \frac{p(n+4)}{2}} 
\int_M |Q_{g_{u_k}}|^p \, u_k^{\frac{2n}{n-4}} \, dV_g\\
& \le \Lambda
\end{align*}
where the last inequality follows from the uniform boundedness of $Q_{g_{u_k}}$in $L^p(M,u_k^\frac{4}{n-4}g)$.
$\hfill$ $\blacksquare$
\noindent
\vspace{0.2cm}
\begin{proposition}\label{prop5}
The sequence $(v_k)_k$ is bounded in the Sobolev space $W^{4,p}_{\mathrm{loc}}(\mathbb{R}^n)$.
\end{proposition}

\proof
Let $R>0$ be fixed. First, we observe that on the Euclidean ball $B(0,R)$, the Sobolev norms associated with the metrics $h_k$ and the Euclidean metric $\delta$ are uniformly equivalent. Indeed, since the metrics $h_k$ converge smoothly to $\delta$ on compact sets, there exists a positive  constant $C=C(R, m, p)$, independent of $k$, such that for any smooth function $u$ supported in $B(0,R)$, for any $m \in \mathbb{N}$ and $p \in [1, \infty[ $,
\[
\frac{1}{C(R)} \|u\|_{W^{m,p}(B(0,R))} \le \|u\|_{W^{m,p}(B(0,R),h_k)} \le C(R) \|u\|_{W^{m,p}(B(0,R))}.
\]
This allows us to work interchangeably with either metric without affecting uniform bounds.

Next, we apply standard elliptic regularity for the Paneitz operator $P_{h_k}$ on $B(0,R)$. More precisely, there exists a constant $C = C(n,R)>0$ such that
\[
\|v_k\|_{W^{4,p}(B(0,\frac{R}{2}),h_k)} \le C \Big( \|P_{h_k}(v_k)\|_{L^{p}(B(0,R ),h_k)} + \|v_k\|_{L^{p}(B(0,R),h_k)} \Big).
\]

By Proposition~\ref{prop4}, the sequence $(P_{h_k}(v_k))_k$ is uniformly bounded in $L^p(B(0,R))$, and since $||v_k||_\infty \le 1$, the $L^p$ norm of $v_k$ is also uniformly bounded. Combining these estimates, we deduce the existence of a constant $C = C(n,R) > 0$, independent of $k$, such that
\[
\|v_k\|_{W^{4,p}(B(0,\frac{R}{2}),h_k)} \le C.
\]

Finally, the uniform equivalence of the Sobolev norms with respect to $h_k$ and $\delta$ implies that $(v_k)_k$ is bounded in $W^{4,p}_{loc}(\mathbb{R}^n)$ with respect to the Euclidean metric.
$\hfill$ $\blacksquare$

\vspace{0.15cm}

\begin{proposition}\label{existence}
There exists a subsequence of \((v_k)_k\), still denoted \((v_k)_k\), and a continuous, non-negative function \(v_\infty : \mathbb{R}^n \to \mathbb{R}\) such that
\[
v_k \longrightarrow v_\infty \quad \text{in } C^{\alpha}_{\mathrm{loc}}(\mathbb{R}^n),
\]
for all $\alpha$ verifying
\[
\alpha \in \left(0,\, \min\Big\{1,\, 4 - \frac{n}{p}\Big\}\right).
\]
\end{proposition}

\proof
By Proposition~\ref{prop5}, the sequence \((v_k)_k\) is bounded in the Sobolev space \(W^{4,p}_{\mathrm{loc}}(\mathbb{R}^n)\). 
For any fixed \(R>0\), the compact Sobolev embedding
\[
W^{4,p}(B(0,R)) \hookrightarrow C^{\alpha}(B(0,R))
\]
holds for all \(\alpha\) in the open interval \(\left(0,\, \min\{1,\, 4 - \frac{n}{p}\}\right)\) (see Theorem 7.26 in \cite{zbMATH01554166}). 
Consequently, we can extract a subsequence, still denoted by \((v_k)_k\), such that
\[
v_k \rightharpoonup v_\infty \quad \text{in } W^{4,p}(B(0,R)) \quad \text{and} \quad v_k \to v_\infty \quad \text{in } C^{\alpha}(B(0,R)).
\]

Applying a diagonal extraction procedure over an increasing sequence of balls \(B(0,R)\), we obtain a function \(v_\infty \in C^{\alpha}_{\mathrm{loc}}(\mathbb{R}^n)\) such that the sequence \((v_k)_k\) converges strongly to \(v_\infty\) in \(C^{\alpha}_{\mathrm{loc}}(\mathbb{R}^n)\).

Finally, since each \(v_k\) is non-negative and uniformly bounded in \(L^\infty(\mathbb{R}^n)\), the limit \(v_\infty\) is also non-negative and satisfies
\[
0 \le v_\infty \le 1, \quad \text{with} \quad v_\infty(0) = 1.
\]
$\hfill$ $\blacksquare$

\subsubsection{Limiting conformal metric }

In this section, we show that the limiting profile \(v_\infty\), obtained through a blow-up analysis, leads to a contradiction with the a priori estimate stated in Proposition~\ref{proposition 1}. This contradiction arises from the preservation of the volume and the uniform lower bound on the first positive eigenvalue of the Laplacian along the sequence.

\begin{proposition}\label{contr1}
The limit function \(v_\infty\) satisfies the identity
\[
\int_{\mathbb{R}^{n}} v_{\infty}^{\frac{2n}{n-4}} \, dx = \sigma_{n}.
\]
\end{proposition}

\proof
Recall that the sequence \((u_k)_k\) of smooth positive functions in the set $\mathcal{E}_{\Lambda,p}$ satisfies
\[
\int_M u_k^{\frac{2n}{n-4}} \, dV_g = \sigma_n, 
\quad  
\lambda_1\big(M,u_{k}^{\frac{4}{n-4}}g\big) \geq n + \frac{1}{\Lambda}.
\]

Then, by Proposition~\ref{prop:volume-inequality}, there exists a constant 
\(C = C(M, g, n) > 0\) such that for all \(0 < r < R\), and for \(k\) sufficiently large so that 
\(\mu_k R < \operatorname{inj}(M)\), we have
\[
\left( \int_{B(x_k, \mu_k r)} u_k^{\frac{2n}{n-4}} \, dV_g \right)^2 
\left( \int_{M \setminus B(x_k, \mu_k R)} u_k^{\frac{2n}{n-4}} \, dV_g \right)
\leq C \left( \log\left( \frac{R}{r} \right) \right)^{\frac{2(1 - n)}{n}}.
\]

Using relation  \eqref{bl trans}, we obtain the corresponding estimate for the function \(v_k\):
\begin{equation}\label{eq:rescaled-volume}
\left( \int_{B(0, \beta_0)} v_k^{\frac{2n}{n-4}} \, dV_{h_k} \right)^2
\left( \sigma_n - \int_{B(0, \beta)} v_k^{\frac{2n}{n-4}} \, dV_{h_k} \right)
\leq C \left( \log\left( \frac{\beta}{\beta_0} \right) \right)^{\frac{2(1 - n)}{n}},
\end{equation}
for every \(0 < \beta_0 < \beta\), and for all \(k\) sufficiently large. Passing to the limit as \(k\) tends to infinity, and using the facts that the sequence of metrics \((h_k)_k\) converges locally to the Euclidean metric \(\delta\) and that \(v_k \to v_\infty\) in \(C^\alpha_{\mathrm{loc}}(\mathbb{R}^n)\), we can apply the dominated convergence theorem to deduce that

\[
\left( \int_{B(0, \beta_0)} v_\infty^{\frac{2n}{n-4}} \, dx \right)^2
\left( \sigma_n - \int_{B(0, \beta)} v_\infty^{\frac{2n}{n-4}} \, dx \right)
\leq C \left( \log\left( \frac{\beta}{\beta_0} \right) \right)^{\frac{2(1 - n)}{n}}.
\]
Since the first factor on the left-hand side is strictly positive, letting  $\beta$ go to infinity  and observing that the right-hand side tends to zero $( n\geq 5 )$, we conclude that
\[
\int_{\mathbb{R}^n} v_\infty^{\frac{2n}{n-4}} \, dx = \sigma_n.
\]
$\hfill$ $\blacksquare$

\begin{proposition}\label{contr2}
Let \( \varphi \in C_c^\infty(\mathbb{R}^n)\setminus\{0\} \) be such that
\[
\int_{\mathbb{R}^n} \varphi\, v_\infty^{\frac{2n}{n-4}}\, dx = 0.
\]
Then the following inequality holds:
\[
\left(n + \frac{1}{\Lambda}\right) \int_{\mathbb{R}^n} \varphi^2\, v_\infty^{\frac{2n}{n-4}}\, dx 
\leq \int_{\mathbb{R}^n} v_\infty^{\frac{2(n-2)}{n-4}}\, |\nabla \varphi|^2\, dx.
\]
\end{proposition}
\proof
Let \( \varphi \in C_c^\infty(\mathbb{R}^n)\setminus\{0\} \), and fix \( 0 < r < \inj(M) \). For \( k \) sufficiently large, the support of \( \varphi \) is contained in the ball \( B(0, r/\mu_k) \), so that \( \varphi \) may be viewed as a function on \( M \) via the rescaled exponential chart. Let \( \eta \in C_c^\infty(\mathbb{R}^n)\setminus\{0\} \) be a cutoff function such that \( \supp(\eta) \subset \supp(\varphi) \) and \( 0 \leq \eta \leq 1 \). Define
\[
c_k := \frac{\int_{\mathbb{R}^n} \varphi\, v_k^{\frac{2n}{n-4}}\, dV_{h_k}}{\int_{\mathbb{R}^n} \eta\, v_k^{\frac{2n}{n-4}}\, dV_{h_k}},
\qquad
\varphi_k := \varphi - c_k \eta,
\]
so that
\[
\int_{\mathbb{R}^n} \varphi_k\, v_k^{\frac{2n}{n-4}}\, dV_{h_k} = 0.
\]

As \( k \to \infty \), the following convergences hold:
\begin{align*}
c_k &\to 0, \\
\int_{\mathbb{R}^n} \varphi_k^2\, v_k^{\frac{2n}{n-4}}\, dV_{h_k} &\to \int_{\mathbb{R}^n} \varphi^2\, v_\infty^{\frac{2n}{n-4}}\, dx, \\
\int_{\mathbb{R}^n} |\nabla_{h_k} \varphi_k|^2_{h_k}\, v_k^{\frac{2(n-2)}{n-4}}\, dV_{h_k} &\to \int_{\mathbb{R}^n} |\nabla \varphi|^2\, v_\infty^{\frac{2(n-2)}{n-4}}\, dx.
\end{align*}

Via the rescaled exponential chart \( \bar{\psi}_k \) from Section~\ref{blup}, we pull back \( \varphi_k \) to a function \\ \( \tilde{\varphi}_k := \varphi_k \circ \bar{\psi}_k^{-1} \) supported in \( B(x_k, r) \subset M \). Then, using relation \eqref{bl trans}, the integrals above become:
\begin{align*}
\int_{\mathbb{R}^n} \varphi_k^2\, v_k^{\frac{2n}{n-4}}\, dV_{h_k} &= \int_{B(x_k, r)} \tilde{\varphi}_k^2\, u_k^{\frac{2n}{n-4}}\, dV_g, \\
\int_{\mathbb{R}^n} |\nabla_{h_k} \varphi_k|^2_{h_k}\, v_k^{\frac{2(n-2)}{n-4}}\, dV_{h_k} &= \int_{B(x_k, r)} |\nabla_g \tilde{\varphi}_k|^2_g\, u_k^{\frac{2(n-2)}{n-4}}\, dV_g.
\end{align*}

Moreover, by construction,
\[
\int_M \tilde{\varphi}_k\, u_k^{\frac{2n}{n-4}}\, dV_g = 0.
\]

Now, using the assumption that 
\(\lambda_1(M,u_k^{\frac{4}{n-4}} g) \geq n + \frac{1}{\Lambda}\), 
and applying the variational characterization of the first positive eigenvalue of the Laplacian on a closed manifold~\eqref{rlq} to \(\tilde{\varphi}_k\), 
we obtain the following inequality:
\[
\left(n + \frac{1}{\Lambda}\right) \int_{B(x_k, r)} \tilde{\varphi}_k^2\, u_k^{\frac{2n}{n-4}}\, dV_g 
\leq \int_{B(x_k, r)} |\nabla_g \tilde{\varphi}_k|_g^2\, u_k^{\frac{2(n-2)}{n-4}}\, dV_g.
\]
which, in rescaled coordinates, becomes
\[
\left(n + \frac{1}{\Lambda}\right) \int_{\mathbb{R}^n} \varphi_k^2\, v_k^{\frac{2n}{n-4}}\, dV_{h_k}
\leq \int_{\mathbb{R}^n} |\nabla_{h_k} \varphi_k|^2_{h_k}\, v_k^{\frac{2(n-2)}{n-4}}\, dV_{h_k}.
\]

Passing to the limit as \( k \to \infty \), we conclude that
\[
\left(n + \frac{1}{\Lambda}\right) \int_{\mathbb{R}^n} \varphi^2\, v_\infty^{\frac{2n}{n-4}}\, dx
\leq \int_{\mathbb{R}^n} |\nabla \varphi|^2\, v_\infty^{\frac{2(n-2)}{n-4}}\, dx.
\]
$\hfill$ $\blacksquare$

\subsubsection{Conclusion of the contradiction argument}

We conclude this section by deriving the desired contradiction. As established in Proposition~\ref{contr1}, there exists a continuous, non-negative function \( v_\infty \) such that
\[
\int_{\mathbb{R}^n} v_\infty^{\frac{2n}{n-4}} \, dx = \sigma_n.
\]
Moreover, applying Proposition~\ref{proposition 1} to the weight \( v_\infty^{\frac{n-2}{n-4}} \), and combining it with Proposition~\ref{contr2}, we obtain the following inequalities:
\[
n + \frac{1}{\Lambda} \leq 
\inf_{\substack{\phi \in C_c^\infty(\mathbb{R}^n) \\ \int_{\mathbb{R}^n} \phi\, v_\infty^{\frac{2n}{n-4}} \, dx = 0}} 
\frac{\displaystyle \int_{\mathbb{R}^n} v_\infty^{\frac{2(n-2)}{n-4}} |\nabla \phi|^2 \, dx}
{\displaystyle \int_{\mathbb{R}^n} v_\infty^{\frac{2n}{n-4}} \phi^2 \, dx}
\leq n,
\]
which is a contradiction. This contradiction shows that the blow-up scenario cannot occur under the given assumptions.
\medskip

\begin{Remark}[The Limiting Equation and the Role of the Spectral Assumption]\label{rq}\leavevmode\par\medskip
Let $v_\infty$ be the function appearing in Proposition~\ref{existence}.
There exists a function \(V \in L^p(\mathbb{R}^n)\) for some $p>\frac{n}{4}$,  such that
\[
\Delta_{\mathrm{Euc}}^{2} v_\infty = V\, v_\infty
\quad \text{for almost every } x \in \mathbb{R}^n,
\]

 where $\Delta_{\mathrm{Euc}}$ denotes the Euclidean Laplacian on $\mathbb{R}^n$. Using this limiting equation, it is possible to construct a sequence of positive functions 
\((u_k)_k\) on \((M,g)\) that blow up, while the associated conformal metrics 
\(
g_k = u_k^{\frac{4}{n-4}} g
\)
satisfy the following properties:
\begin{itemize}
    \item the \(Q\)-curvature of \(g_k\) remains uniformly bounded in \(L^p(M,g_k)\) .
    \item the volume of \(M\) with respect to \(g_k\) is a positive constant 
    independent of \(k\).
\end{itemize}
Such a construction illustrates the importance of the assumption on the first positive 
eigenvalue of the Laplacian: without this spectral hypothesis, 
blow-up phenomena may occur despite uniform $L^p$-control of the 
$Q$-curvature and fixed total volume.
\end{Remark}
\subsection{Sobolev Boundedness}
\begin{proposition}\label{boundedness}
The set $\mathcal{E}_{\Lambda,p}$ is bounded in the Sobolev space $W^{4,p}(M)$.
\end{proposition}

\proof
Let $u\in \mathcal{E}_{\Lambda,p}$ and set $g_u:=u^{\frac{4}{n-4}}g$. By standard elliptic regularity for the Paneitz operator, there exists a constant $C_0=C_0(M,n,g)>0$ such that
\begin{equation}\label{eq:ell-reg}
\|u\|_{W^{4,p}(M)} \le C_0\big(\|P_g u\|_{L^p(M)}+\|u\|_{L^p(M)}\big).
\end{equation}
Since $u$ satisfies
\[
P_g u = Q_{g_u}\, u^{\frac{n+4}{n-4}},
\]
it remains to estimate the $L^p$--norm of the right-hand side. We compute
\[
\|Q_{g_u}\, u^{\frac{n+4}{n-4}}\|_{L^p(M)}^p
= \int_M |Q_{g_u}|^p \, u^{p\frac{n+4}{n-4}} \,dV_g
= \int_M |Q_{g_u}|^p \, u^{\frac{2n}{n-4}} \, u^{\frac{p(n+4)-2n}{n-4}} \,dV_g.
\]
Using that $u$ is uniformly bounded on $\mathcal{E}_{\Lambda,p}$ (say $\|u\|_{L^\infty(M)}\le B$ for some $B>0$) and the hypothesis
\[
\int_M |Q_{g_u}|^p\,u^{\frac{2n}{n-4}}\,dV_g \le \Lambda,
\]
we obtain
\[
\|Q_{g_u}\, u^{\frac{n+4}{n-4}}\|_{L^p(M)}^p
\le \Lambda \, B^{\frac{p(n+4)-2n}{n-4}}.
\]
Taking $p$-th roots yields the uniform bound
\[
\|P_g u\|_{L^p(M)} = \|Q_{g_u}\, u^{\frac{n+4}{n-4}}\|_{L^p(M)}
\le \Lambda^{1/p}\, B^{\frac{p(n+4)-2n}{p(n-4)}}.
\]
Moreover, since $M$ is compact and $\|u\|_{L^\infty(M)}\le B$, the $L^p$-norm of $u$ is uniformly bounded by $B\,\mathrm{Vol}(M)^{1/p}$. Inserting these bounds into \eqref{eq:ell-reg} gives
\[
\|u\|_{W^{4,p}(M)} \le C_1,
\]
for some constant $C_1=C_1(M,g,n,p,\Lambda,B)>0$ independent of $u\in \mathcal{E}_{\Lambda,p}$. This proves that $\mathcal{E}_{\Lambda,p}$ is bounded in $W^{4,p}(M)$.
$\hfill$ $\blacksquare$

\subsection{Precompacity in  H\"older function space  }
\begin{proposition}\label{wal shi}
Let \((M^n, g)\) be a closed Riemannian manifold of dimension \(n \geq 5.\) Let \(p > \frac{n}{4}\) and \(0 < \alpha < \min\left\{1, 4 - \frac{n}{p} \right\}\). The set \(\mathcal{E}_{\Lambda, p}\) given in \eqref{function class} is precompact in the Hölder space \(C^{\alpha}(M)\).
\end{proposition}

\proof
The conclusion follows from two facts. First, the Sobolev embedding
\[
W^{4,p}(M) \hookrightarrow C^{\alpha}(M)
\]
is compact under the assumptions on \(p\) and \(\alpha\) (see part (b) of Theorem 2.2 in \cite{zbMATH04030435}). Second, the set \(\mathcal{E}_{\Lambda,p}\) is uniformly bounded in \(W^{4,p}(M)\) (see Proposition~\ref{boundedness}). The compactness of the embedding implies that any bounded sequence in \(\mathcal{E}_{\Lambda,p}\) has a convergent subsequence in \(C^\alpha(M)\), which proves the precompactness.
$\hfill$ $\blacksquare$
\section{Proof of Theorem 2}\label{sec3}

\subsection{Positive Lower bound and precompactness in the H\"older metric space}
 In this paper, we are working with a metric norm; that is why, if we search for a compactness result for the set $E_{\Lambda,p}$, we should check that a sequence in $\mathcal{E}_{\Lambda,p}$ admits a positive uniform lower bound. For this, we will prove some Sobolev estimates, and we will use the sign hypothesis on the scalar curvature and the Q-curvature to construct this positive uniform lower bound.
\begin{proposition}\label{jkl}
Let \((M^n, g)\) be a closed smooth Riemannian manifold of dimension \(n \geq 5\). Suppose that 
\begin{itemize}
    \item the scalar curvature satisfies \(S_g \geq 0\),
    \item the Q-curvature is semi-positive; that is, \(Q_g \geq 0\) and \(Q_g > 0\) somewhere on \(M\).
\end{itemize}
Then there exists a positive constant \(C =(g,\Lambda,n,p)\) such that
\begin{equation} \label{eq:positive-energy}
\int_M u\,P_g u \,dV_g \;\ge\; C
\qquad \text{for every } u \in \mathcal{E}_{\Lambda,p}.
\end{equation}
\end{proposition}

\proof
By a result of Gursky–Malchiodi (see Proposition 2.3 in \cite{gursky2015strong}), the sign assumptions on the scalar curvature and the Q-curvature imply the existence of a positive constant \(C = C(g)\) such that
\begin{equation} \label{eq:gursky-malchiodi}
\int_M u\, P_g u\, dV_g \;\ge\; C \int_M u^2\, dV_g \quad \text{for all } u \in C^\infty(M).
\end{equation}

On the other hand, the \(L^2\)-norm of \(u\) can be estimated from below using Hölder's inequality:
\begin{equation} \label{eq:L2-vs-Lq}
\int_M u^2\, dV_g \;\ge\; \|u\|_{L^\infty(M)}^{-\frac{8}{n-4}} \int_M u^{\frac{2n}{n-4}}\, dV_g.
\end{equation}

Since the family \(\mathcal{E}_{\Lambda, p}\) is uniformly bounded in \(L^\infty(M)\) (see section \ref{blup}), and since all elements of \(\mathcal{E}_{\Lambda, p}\) satisfy the fixed volume constraint \eqref{volume_cndt}, it follows from \eqref{eq:gursky-malchiodi} and \eqref{eq:L2-vs-Lq} that there exists a positive constant $C =(g,\Lambda,n,p)$ such that 
\[
\int_M u\,P_g u \,dV_g \;\ge\; C
\qquad \text{for every } u \in \mathcal{E}_{\Lambda,p}.
\]
$\hfill$ $\blacksquare$

\begin{proposition}
Let $(M, g)$ be a closed smooth Riemannian manifold satisfying the hypotheses of Proposition~\ref{jkl}. Then there exists a positive constant \(C =(M,g,\Lambda,n,p) \) such that
\[
\inf_{x \in M} u(x) \geq C \quad \text{for all } u \in \mathcal{E}_{\Lambda, p}.
\]
\end{proposition}

\begin{proof}
By the results of Gursky and Malchiodi (see Proposition 2.3 in \cite{gursky2015strong}), the positivity assumptions on the scalar curvature \(S_g\) and the \(Q\)-curvature \(Q_g\) ensure that the Paneitz operator \(P_g\) is invertible and admits a Green's function \(G_P(x, y)\). Moreover, as shown in inequality (2.3) of \cite{Gangli}, there exists a positive constant \(C = C(M, n, g) \) such that
\[
G_P(x, y) \ge C \quad \text{for all } x \neq y \in M.
\]

For any \(u \in \mathcal{E}_{\Lambda, p}\), Green's representation formula gives
\[
u(x) = \int_M G_P(x, y) \, P_g u(y) \, dV_g(y).
\]
Since \(P_g u \ge 0\) and \(G_P(x, y) \ge C\) for all \(x \neq y \in M\), it follows that
\begin{equation}\label{abc}
u(x)
\ge C \int_M P_g u(y) \, dV_g(y).
\end{equation}
Now we estimate the right hand side of the inequality \eqref{abc} as follows:
\begin{align*}
u(x) 
&\ge C \int_M P_g u(y) \, dV_g(y) \\
&= C \int_M \frac{u(y)}{u(y)} \, P_g u(y) \, dV_g(y) \\
&\ge C \, \|u\|_{L^\infty(M)}^{-1} \int_M u(y) \, P_g u(y) \, dV_g(y) \\[1mm]
&\ge C.
\end{align*}
where the last inequality follows from the uniform \(L^\infty\)-boundedness of functions in \(\mathcal{E}_{\Lambda, p}\) and Proposition~\ref{jkl}. Hence, the infimum of \(u\) over \(M\) is uniformly bounded below by a positive constant independent of \(u \in \mathcal{E}_{\Lambda, p}\), as claimed.
\end{proof}
Once the non-degeneration phenomenon for the family of metrics in \(E_{\Lambda,p}\) has been established, the precompactness of \(\mathcal{E}_{\Lambda,p}\) in \({C}^\alpha(M)\) (see Proposition~\ref{wal shi}) implies the precompactness of $E_{\Lambda,p}$ in $\mathscr{C}^{\alpha}(M)$.
 
\begin{proposition}
    Let \((M^n, g)\) be a closed Riemannian manifold of dimension \(n \geq 5\), and assume the hypotheses of Proposition~\ref{jkl} hold. Let \(p > \frac{n}{4}\) and \(0 < \alpha < \min\left\{1, 4 - \frac{n}{p} \right\}\). The set \(E_{\Lambda, p}\) is precompact in the Hölder metric  space \(\mathscr{C}^{\alpha}(M)\).
\end{proposition}

\medskip

\begin{Remark}[On the Necessity of the Non-negativity Assumption on \(P_g(u)\)]
In this Remark, we construct a counterexample that demonstrates the condition
\[
P_g(u) \ge 0  \text{ for } u  \in  \mathcal{E}_{\Lambda,p} ,
\]
is essential to guaranty compactness in the \(Q\)-curvature framework.
Indeed, this assumption prevents degeneracy phenomena in the limiting metrics
and, in particular, rules out the collapse of the conformal factor in blow-up
analysis.

\begin{proof}
Let \((\mathbb{S}^n,\bar{g})\) denote the standard unit sphere with \(n \ge 5\), and fix a point
\(x_0 \in \mathbb{S}^n\). Choose
\[
\frac{n}{4} < p < \frac{n^2}{2(n+4)}.
\]
Recall that the function
\[
\eta_{\bar{g}}(x) := \cos \bigl(d_{\bar{g}}(x,x_0)\bigr)
\]
is an eigenfunction of \(\Delta_{\bar{g}}\) with respect to the round metric. For any \(0 < \varepsilon < 1\), define the smooth positive function
\[
u_\varepsilon := 1 - \varepsilon \eta_{\bar{g}}.
\]

The  family $(u_\epsilon)_\epsilon $ satisfies the following properties:
\begin{itemize}
\item $u_\varepsilon > 0$ on $\mathbb{S}^n$;
\item $u_\varepsilon(x_0) \to 0$ as $\varepsilon \to 1$;
\item $u_\varepsilon$ is uniformly bounded in $L^\infty(\mathbb{S}^n)$, namely
\(
\|u_\varepsilon\|_{L^\infty(\mathbb{S}^n)} \leq 2;
\)
\item the volume of the conformal metric
\(
g_\varepsilon := u_\varepsilon^{\frac{4}{n-4}} \bar{g}\)
is uniformly bounded from below, since
\[
\int_{\mathbb{S}^n} u_\varepsilon^{\frac{2n}{n-4}}\, dV_{\bar{g}}
\ge \biggl(\frac{1}{2}\biggr)^{\frac{2n}{n-4}}
\mathrm{Vol}_{\bar{g}}\bigl(\mathbb{S}^n \setminus B_{\bar{g}}(x_0,\pi/3)\bigr).
\]
\end{itemize}

We now examine the \(Q\)-curvature of the metric \(g_\varepsilon\).
On the round sphere, the Paneitz operator takes the form
\[
P_{\bar{g}} = \Delta_{\bar{g}}^2 + a_n \Delta_{\bar{g}} + b_n,
\]
where $a_n$ and $b_n$ are two positive constants depending only on \(n\). On the other hand, since \(\eta_{\bar{g}}\) is an eigenfunction, we have
\[
P_{\bar{g}}(\eta_{\bar{g}}) = c_n \eta_{\bar{g}},\qquad
c_n := n^2 + n a_n + b_n.
\]
Hence,
\[
P_{\bar{g}}(u_\varepsilon)
= c_n\, u_\varepsilon - n(n+a_n).
\]

Using the conformal transformation law for the \(Q\)-curvature, we find
\[
Q_{g_\varepsilon}
= c_n\, u_\varepsilon^{-\frac{8}{n-4}}
- n(n +a_n)\, u_\varepsilon^{-\frac{n+4}{n-4}}.
\]
Thus, the potential blow-up of \(Q_{g_\varepsilon}\) is driven by the singular
behavior of \(u_\varepsilon^{-1}\) near the point \(x_0\).

We now prove that \(Q_{g_\varepsilon}\) remains uniformly bounded in
\(L^p(\mathbb{S}^n,g_\varepsilon)\). Introducing polar coordinates
\(\theta := d_{\bar{g}}(x,x_0)\), one obtains
\begin{align*}
\int_{\mathbb{S}^n}
u_\varepsilon^{\frac{2n - p(n+4)}{n-4}}\, dV_{\bar{g}}
&= \int_{\mathbb{S}^n}
\bigl(1 - \varepsilon \eta_{\bar{g}}\bigr)^{\frac{2n - p(n+4)}{n-4}}
\, dV_{\bar{g}} \\
&\le C \int_0^\pi
\bigl(1 - \cos\theta\bigr)^{\frac{2n - p(n+4)}{n-4}}
\sin^{n-1}\theta\, d\theta \\
&\le C \int_0^\pi
\sin^{\frac{4n - 2p(n+4)}{n-4}}\!\Bigl(\frac{\theta}{2}\Bigr)
\sin^{n-1}\theta\, d\theta,
\end{align*}
which is finite under the imposed condition on \(n\) and \(p\).
A similar estimate shows that
\(u_\varepsilon^{-\frac{8}{n-4}} \) is  uniformly bounded  \(L^p(\mathbb{S}^n,g_\varepsilon)\).
Consequently, \(\|Q_{g_\varepsilon}\|_{L^p(\mathbb{S}^n,g_\varepsilon)}\)
remains uniformly bounded in \(L^p(\mathbb{S}^n,g_\varepsilon)\).

Nevertheless, the conformal factor \(u_\varepsilon\) collapses at \(x_0\),
despite bounded \(L^p\)-curvature and a uniform lower volume bound.
This shows that compactness fails without the assumption \(P_g(u)\ge 0\).

The non-negativity condition on the Paneitz operator is, therefore, a necessary hypothesis to prevent such degeneracy.
\end{proof}

\end{Remark}

\bibliographystyle{siam}
\bibliography{references}

@article{zbMATH03960440,
 author = {Branson, Thomas P.},
 title = {Differential operators canonically associated to a conformal structure},
 fjournal = {Mathematica Scandinavica},
 journal = {Math. Scand.},
 issn = {0025-5521},
 volume = {57},
 pages = {293--345},
 year = {1985},
 language = {English},
 doi = {10.7146/math.scand.a-12120},
 url = {https://eudml.org/doc/166957},
 zbMATH = {3960440},
 Zbl = {0596.53009}
}

@article{zbMATH00027724,
  author  = {Li, Peter and Yau, Shing-Tung},
  title  =  {A New Conformal Invariant and Its Applications to the {W}illmore Conjecture and the First Eigenvalue of Compact Surfaces},
  journal = {Invent. Math.},
  volume  = {69},
  number  = {2},
  pages   = {269--291},
  year    = {1982}
}

@book{zbMATH01554166,
 author = {Gilbarg, David and Trudinger, Neil S.},
 title = {Elliptic partial differential equations of second order},
 edition = {Reprint of the 1998 ed.},
 fseries = {Classics in Mathematics},
 series = {Class. Math.},
 issn = {1431-0821},
 isbn = {3-540-41160-7},
 year = {2001},
 publisher = {Berlin: Springer},
 language = {English},
 zbMATH = {1554166},
 Zbl = {1042.35002}
}

@article{zbMATH04030435,
 author = {Lee, John M. and Parker, Thomas H.},
 title = {The {Yamabe} problem},
 fjournal = {Bulletin of the American Mathematical Society. New Series},
 journal = {Bull. Am. Math. Soc., New Ser.},
 issn = {0273-0979},
 volume = {17},
 pages = {37--91},
 year = {1987},
 language = {English},
 doi = {10.1090/S0273-0979-1987-15514-5},
 zbMATH = {4030435},
 Zbl = {0633.53062}
}

@article{gursky2015strong,
  title={A strong maximum principle for the Paneitz operator and a non-local flow for the {Q}-curvature},
  author={Gursky, Matthew J and Malchiodi, Andrea},
  journal={Journal of the European Mathematical Society (EMS Publishing)},
  volume={17},
  number={9},
  year={2015}
}

@article{zbMATH00558966,
 author = {Gursky, Matthew J.},
 title = {Compactness of conformal metrics with integral bounds on curvature},
 fjournal = {Duke Mathematical Journal},
 journal = {Duke Math. J.},
 issn = {0012-7094},
 volume = {72},
 number = {2},
 pages = {339--367},
 year = {1993},
 language = {English},
 doi = {10.1215/S0012-7094-93-07212-2},
 url = {thesis.library.caltech.edu/2650/9/Gursky_mj_1991.pdf},
 zbMATH = {558966},
 Zbl = {0809.53040}
}

@article{zbMATH06774840,
 author = {Matthiesen, Henrik},
 title = {Regularity of conformal metrics with large first eigenvalue},
 fjournal = {Annales de la Facult{\'e} des Sciences de Toulouse. Math{\'e}matiques. S{\'e}rie VI},
 journal = {Ann. Fac. Sci. Toulouse, Math. (6)},
 issn = {0240-2963},
 volume = {25},
 number = {5},
 pages = {1079--1094},
 year = {2016},
 language = {English},
 doi = {10.5802/afst.1523},
 url = {semanticscholar.org/paper/1bef0a609635c287aa9da92691bbe11692d9ae86},
 zbMATH = {6774840},
 Zbl = {1373.53052}
}

@book{zbMATH05835819,
 author = {de Saint-Gervais, Henri Paul},
 title = {Uniformisation des surfaces de {Riemann}. {Retour} sur un th{\'e}or{\`e}me centenaire},
 isbn = {978-2-84788-233-9},
 year = {2010},
 publisher = {Lyon: ENS {\'E}ditions},
 language = {French},
 zbMATH = {5835819},
 Zbl = {1228.30001}
}

@article{zbMATH05503073,
 author = {Khuri, M. A. and Marques, F. C. and Schoen, R. M.},
 title = {A compactness theorem for the {Yamabe} problem},
 fjournal = {Journal of Differential Geometry},
 journal = {J. Differ. Geom.},
 issn = {0022-040X},
 volume = {81},
 number = {1},
 pages = {143--196},
 year = {2009},
 language = {English},
 doi = {10.4310/jdg/1228400630},
 zbMATH = {5503073},
 Zbl = {1162.53029}
}

@article{zbMATH02220860,
 author = {Li, YanYan and Zhang, Lei},
 title = {Compactness of solutions to the {Yamabe} problem. {II}},
 fjournal = {Calculus of Variations and Partial Differential Equations},
 journal = {Calc. Var. Partial Differ. Equ.},
 issn = {0944-2669},
 volume = {24},
 number = {2},
 pages = {185--237},
 year = {2005},
 language = {English},
 doi = {10.1007/s00526-004-0320-7},
 zbMATH = {2220860},
 Zbl = {1229.35071}
}

@article{paneitz2008quartic,
   title={A Quartic Conformally Covariant Differential Operator for Arbitrary Pseudo-Riemannian Manifolds (Summary)},
   ISSN={1815-0659},
   url={http://dx.doi.org/10.3842/SIGMA.2008.036},
   DOI={10.3842/sigma.2008.036},
   journal={Symmetry, Integrability and Geometry: Methods and Applications},
   publisher={SIGMA (Symmetry, Integrability and Geometry: Methods and Application)},
   author={Paneitz, Stephen},
   year={2008},
   month=mar }

@article{yanyanli,
 author = {Li, YanYan and Xiong, Jingang},
 title = {Compactness of conformal metrics with constant {{\(Q\)}}-curvature. {I}},
 fjournal = {Advances in Mathematics},
 journal = {Adv. Math.},
 issn = {0001-8708},
 volume = {345},
 pages = {116--160},
 year = {2019},
 language = {English},
 doi = {10.1016/j.aim.2019.01.020},
 zbMATH = {7021539},
 Zbl = {1484.53075}
}

@article{Gangli,
 author = {Li, Gang},
 title = {A compactness theorem on {Branson}'s {{\(Q\)}}-curvature equation},
 fjournal = {Pacific Journal of Mathematics},
 journal = {Pac. J. Math.},
 issn = {1945-5844},
 volume = {302},
 number = {1},
 pages = {119--179},
 year = {2019},
 language = {English},
 doi = {10.2140/pjm.2019.302.119},
 zbMATH = {7178918},
 Zbl = {1434.53041}
}

@article{zbMATH05817575,
 author = {Brendle, Simon},
 title = {Blow-up phenomena for the {Yamabe} equation},
 fjournal = {Journal of the American Mathematical Society},
 journal = {J. Am. Math. Soc.},
 issn = {0894-0347},
 volume = {21},
 number = {4},
 pages = {951--979},
 year = {2008},
 language = {English},
 doi = {10.1090/S0894-0347-07-00575-9},
 zbMATH = {5817575},
 Zbl = {1206.53041}
}

@article{zbMATH05522838,
 author = {Brendle, Simon and Marques, Fernando C.},
 title = {Blow-up phenomena for the {Yamabe} equation. {II}},
 fjournal = {Journal of Differential Geometry},
 journal = {J. Differ. Geom.},
 issn = {0022-040X},
 volume = {81},
 number = {2},
 pages = {225--250},
 year = {2009},
 language = {English},
 doi = {10.4310/jdg/1231856261},
 zbMATH = {5522838},
 Zbl = {1166.53025}
}

@article{zbMATH02207680,
 author = {Druet, Olivier},
 title = {Compactness for {Yamabe} metrics in low dimensions},
 fjournal = {IMRN. International Mathematics Research Notices},
 journal = {Int. Math. Res. Not.},
 issn = {1073-7928},
 volume = {2004},
 number = {23},
 pages = {1143--1191},
 year = {2004},
 language = {English},
 doi = {10.1155/S1073792804133278},
 zbMATH = {2207680},
 Zbl = {1085.53029}
}

@article{zbMATH05033801,
 author = {Marques, Fernando Coda},
 title = {A priori estimates for the {Yamabe} problem in the non-locally conformally flat case},
 fjournal = {Journal of Differential Geometry},
 journal = {J. Differ. Geom.},
 issn = {0022-040X},
 volume = {71},
 number = {2},
 pages = {315--346},
 year = {2005},
 language = {English},
 doi = {10.4310/jdg/1143651772},
 zbMATH = {5033801},
 Zbl = {1101.53019}
}

@article{zbMATH05161231,
 author = {Li, Yanyan and Zhang, Lei},
 title = {Compactness of solutions to the {Yamabe} problem. {III}},
 fjournal = {Journal of Functional Analysis},
 journal = {J. Funct. Anal.},
 issn = {0022-1236},
 volume = {245},
 number = {2},
 pages = {438--474},
 year = {2007},
 language = {English},
 doi = {10.1016/j.jfa.2006.11.010},
 zbMATH = {5161231},
 Zbl = {1229.35072}
}

@misc{zbMATH03337135,
 author = {Cheeger, J.},
 title = {A lower bound for the smallest eigenvalue of the {Laplacian}},
 year = {1970},
 language = {English},
 howpublished = {Probl. {Analysis}, {Sympos}. in {Honor} of {Salomon} {Bochner}, {Princeton} {Univ}. 1969, 195-199 (1970).},
 zbMATH = {3337135},
 Zbl = {0212.44903}
}

@article{trudinger1968remarks,
  title={Remarks concerning the conformal deformation of Riemannian structures on compact manifolds},
  author={Trudinger, Neil S},
  journal={Annali della Scuola Normale Superiore di Pisa-Scienze Fisiche e Matematiche},
  volume={22},
  number={2},
  pages={265--274},
  year={1968}
}

\end{document}